\documentclass[11pt]{amsart}
\usepackage{amsmath, amsthm, amssymb, hyperref, geometry, url, mathrsfs}

\evensidemargin \oddsidemargin
\author{Benjamin Linowitz}
\address{Department of Mathematics\\University of Michigan\\Ann Arbor, MI 48109}
\email{linowitz@umich.edu}

\author{Matthew Stover}
\address{Department of Mathematics, Temple University, 1805 N.\ Broad Street, Philadelphia, PA 19122, USA}
\email{mstover@temple.edu}

\title{Parametrizing Shimura subvarieties of $\mathrm{A}_1$ Shimura varieties and related geometric problems}

\usepackage{fancyhdr}

\fancypagestyle{plain}

\DeclareMathAlphabet{\curly}{U}{rsfs}{m}{n}

\DeclareMathOperator{\Ann}{Ann}

\DeclareMathOperator{\Aut}{Aut}

\DeclareMathOperator{\Ram}{Ram}

\DeclareMathOperator{\Gal}{Gal}
\DeclareMathOperator{\GL}{GL}

\DeclareMathOperator{\inv}{inv}

\DeclareMathOperator{\SL}{SL}

\newtheorem{thm}{Theorem}[section]
\newtheorem{cor}[thm]{Corollary}

\newtheorem{lem}[thm]{Lemma}
\theoremstyle{definition}

\newtheorem*{rmk}{Remark}

\theoremstyle{remark}

\newtheorem{proposition}{Proposition}[section]
\def\1{\mathbf{1}}

\theoremstyle{plain}

\theoremstyle{remark}

\newtheorem{example}[proposition]{Example}

\def\1{\mathbf{1}}

\def\Gal{\mathrm{Gal}}

\newcommand{\bfC}{\mathbb C}

\newcommand{\bfQ}{\mathbb Q}
\newcommand{\bfR}{\mathbb R}

\newcommand{\bfZ}{\mathbb Z}

\newcommand{\frakp}{\mathfrak{p}}
\newcommand{\frakq}{\mathfrak{q}}

\newcommand{\frakP}{\mathfrak{P}}

\makeatletter
\def\moverlay{\mathpalette\mov@rlay}
\def\mov@rlay#1#2{\leavevmode\vtop{%
   \baselineskip\z@skip \lineskiplimit-\maxdimen
   \ialign{\hfil$\m@th#1##$\hfil\cr#2\crcr}}}
\newcommand{\charfusion}[3][\mathord]{
    #1{\ifx#1\mathop\vphantom{#2}\fi
        \mathpalette\mov@rlay{#2\cr#3}
      }
    \ifx#1\mathop\expandafter\displaylimits\fi}
\makeatother

\makeatletter
\let\@@pmod\pmod
\DeclareRobustCommand{\pmod}{\@ifstar\@pmods\@@pmod}
\def\@pmods#1{\mkern4mu({\operator@font mod}\mkern 6mu#1)}
\makeatother

\begin{document}

\begin{abstract}
This paper gives a complete parametrization of the commensurability classes of totally geodesic subspaces of irreducible arithmetic quotients of $X_{a, b} = (\mathbf{H}^2)^a \times (\mathbf{H}^3)^b$. A special case describes all Shimura subvarieties of type $\mathrm{A}_1$ Shimura varieties. We produce, for any $n\geq 1$, examples of manifolds/Shimura varieties with precisely $n$ commensurability classes of totally geodesic submanifolds/Shimura subvarieties. This is in stark contrast with the previously studied cases of arithmetic hyperbolic $3$-manifolds and quaternionic Shimura surfaces, where the presence of one commensurability class of geodesic submanifolds implies the existence of infinitely many classes.
\end{abstract}

%------------------------------------------------------------------------------
%------------------- Front data -----------------------------------------------
\maketitle

%\tableofcontents{}

\vspace{-2pc}

%---------------------------------------------------------------------------------
%%% Introduction
%---------------------------------------------------------------------------------
\section{Introduction}

It is well-known that there is a close relationship between irreducible finite-volume quotients of the global symmetric space $X_{a, b} = (\mathbf{H}^2)^a \times (\mathbf{H}^3)^b$, where $\mathbf{H}^d$ denotes $d$-dimensional hyperbolic space, and quaternion algebras over number fields satisfying certain conditions. For example, see \cite{Vigneras}. In this paper we study the parametrization of nonabelian subalgebras of quaternion algebras over number fields and its effects on the geometry of the associated quotient spaces.

Previous work on this problem was restricted to the cases of \emph{arithmetic Kleinian groups} ($a = 0$ and $b = 1$) by Maclachlan--Reid \cite{MRFuchsian} or \cite[\S 9.5]{MR} and \emph{quaternionic Shimura surfaces} ($a = 2$ and $b = 0$) by Chinburg and the second author \cite{CS}. In both cases the proper closed geodesic subspaces associated with quaternion subalgebras are immersed Riemann surfaces, and one sees that if an irreducible arithmetic quotient of $X_{a, b}$ contains one such geodesic subspace, then it contains infinitely many. Moreover, the collection of all such Riemann surfaces forms infinitely many distinct commensurability classes; recall that two Riemannian manifolds are commensurable if they share an isometric finite-sheeted covering. We will show that this dichotomy of `none or infinitely many commensurability classes' does not necessarily persist for all $X_{a, b}$. For the main technical result, see Theorem \ref{theorem:sublattices}.

First, we describe some positive consequences. For both arithmetic Kleinian groups and quaternionic Shimura surfaces with invariant trace field $K$, closed geodesic subspaces are associated with quaternion subalgebras defined over quadratic subfields $K_0 \subset K$. In this case $K / K_0$ is necessarily Galois with Galois group $\bfZ / 2$. We show that this Galois condition is in fact the driving force behind the above dichotomy.

%---------------------------------------------------------------------------------
{
\renewcommand{\thethm}{\ref{theorem:infinitelymanysublattices}}
\begin{thm}
Let $K$ be a number field and $K_0\subset K$ a subfield for which $K/K_0$ is Galois of even degree. If $\Gamma\subset \SL_2(\bfR)^a\times\SL_2(\bfC)^b$ is an irreducible arithmetic lattice arising from a quaternion algebra $A$ over $K$ that contains an irreducible arithmetic sublattice $\Sigma\subset \SL_2(\bfR)^c\times \SL_2(\bfC)^d$ arising from a quaternion algebra $B$ over $K_0$, then $\Gamma$ contains infinitely many pairwise incommensurable arithmetic lattices in $\SL_2(\bfR)^c\times \SL_2(\bfC)^d$.
\end{thm}
\addtocounter{thm}{-1}
}
%---------------------------------------------------------------------------------

For example, this always holds when $K / \bfQ$ is cyclic of even degree and $A = B \otimes_{\bfQ} K$ for some $\bfQ$-quaternion algebra $B$ (cf.\ Corollary \ref{cor:2^n}). However, such a result does not hold in general. When the extension $K/K_0$ is either not Galois or is of odd degree, new possible finiteness phenomena arise. For instance in Theorem \ref{thm:A5} we prove the existence of arithmetic lattices acting on $(\mathbf{H}^2)^6$ containing a unique commensurability class of arithmetic Fuchsian subgroups and no other arithmetic sublattices. In light of our results it seems likely that the `nonempty implies infinitely many commensurability classes' phenomenon is unique to the cases where $X_{a, b}$ has dimension at most four, i.e., hyperbolic $3$-manifolds and quaternionic Shimura surfaces. We note that similar results have been obtained by McReynolds \cite{McR}. 

The following result exemplifies the behavior that can arise when $K / K_0$ has odd degree.

%---------------------------------------------------------------------------------
{
\renewcommand{\thethm}{\ref{theorem:finitelymany}}
\begin{thm}
Let $K/\bfQ$ be a totally real cyclic Galois extension of odd degree $n>1$. There exist infinitely many pairwise incommensurable irreducible lattices $\Gamma\subset \SL_2(\bfR)^n$ with invariant trace field $K$ whose arithmetic sublattices lie in precisely $\tau(n)$ commensurability classes, where $\tau$ denotes the divisor function.
\end{thm}
\addtocounter{thm}{-1}
}
%---------------------------------------------------------------------------------

Note that $\tau(n)$ achieves every positive integer value. We emphasize that, even when there are finitely many commensurability classes, each commensurability class determines infinitely many distinct immersed submanifolds. Although this is well-known to experts, we give a proof in Lemma \ref{lem:LotsOfDistinct} for the convenience of the reader.

When $K$ is a totally real field (i.e., $b = 0$) we are considering arithmetic quotients of $(\mathbf{H}^2)^a$. These are also known as \emph{Shimura varieties of type} $\mathrm{A}_1$. In particular, certain special elements of each commensurability class, namely those coming from principal congruence subgroups, parametrize principally polarized abelian varieties with added structure. See \cite[p.~87]{Milne}. Moreover, the Shimura subvarieties are totally geodesic submanifolds of complex codimension $0 < c < a$. The fact that $K$ is totally real makes the classification simpler in this case.

%---------------------------------------------------------------------------------
\begin{thm}\label{thm:Shimura}
Let $V$ be a Shimura variety of type $\mathrm{A}_1$ with associated totally real field $K$ and $K$-quaternion algebra $A$. Suppose that $V$ is the quotient of $(\mathbf{H}^2)^a$ by an irreducible lattice in $\SL_2(\bfR)^a$, so $\dim_\bfC(V) = a$.
\begin{enumerate}

\item Necessary conditions for $V$ to contain a proper Shimura subvariety of dimension $0 < c < a$ are that $c \mid a$ and for there to exist a proper subfield $K_0 \subset K$ satisfying
\[
[K : K_0] = \frac{a}{c}.
\]

\item Suppose that (i) holds. Then $K_0$ determines a nonempty set of Shimura subvarieties of $V$ if and only if for every $w \in \Ram(A)$ and $v$ the unique place of $K_0$ divisible by $w$, the following hold:
\begin{enumerate}
\item The local degree $[K_w : (K_0)_v]$ is odd.
\item If $w'\neq w$ is place of $K$ dividing $v$ and the local degree $[K_{w'} : (K_0)_v]$ is odd then $w'\in\Ram(A)$.
\item If every place $v'$ of $K_0$ is divisible by a place $w'$ of $K$ for which $[K_{w'} : (K_0)_{v'}]$ is odd, then \[\{v\ :\ v \text{ is a place of $K_0$ divisible by some } w\in\Ram(A)\}\] has even cardinality.
\end{enumerate}

\item Suppose that (i) and (ii) hold and that $K / K_0$ is Galois of even degree. Then $V$ contains infinitely many distinct commensurability classes of distinct Shimura subvarieties arising from $K_0$-subalgebras of $A$.

\item Suppose that (i) holds and that $K / K_0$ has odd degree or is not Galois. Then $K_0$ could determine any number $0 \le t \le \infty$ of distinct commensurability classes of Shimura subvarieties of $V$. Moreover, every value of $t$ is attained by infinitely many incommensurable Shimura varieties.

\end{enumerate}
\end{thm}
%---------------------------------------------------------------------------------

The ingredients for the proof of Theorem \ref{thm:Shimura} are a combination of various results below, which we quickly assemble.

%---------------------------------------------------------------------------------
\begin{proof}[Proof of Theorem \ref{thm:Shimura}]
Part (i) follows from Theorem \ref{theorem:sublattices}. See Example \ref{example:totallyreal} for details in the totally real case. Part (ii) is Lemma \ref{lem:WhenEmbed?}, and (iii) is Theorem \ref{theorem:infinitelymanysublattices}. The techniques for computing the number $t$ in case (iv) of Theorem \ref{thm:Shimura}, and examples realizing each value, are described throughout the paper. For a complete analysis of the odd degree case, one follows the proof of Theorem \ref{theorem:finitelymany} and Lemma \ref{lem:WhenEmbed?}. %Example \ref{ex:A5} provides the method for constructing examples when $K / K_0$ is Galois of even degree.
\end{proof}
%---------------------------------------------------------------------------------

\noindent
\textbf{Acknowledgments} The first author was partially supported by an NSF RTG grant DMS-1045119 and an NSF Mathematical Sciences Postdoctoral Fellowship. The second author was supported by the National Science Foundation under Grant Number NSF DMS-1361000 and acknowledges support from U.S. National Science Foundation grants DMS 1107452, 1107263, 1107367 ``RNMS: GEometric structures And Representation varieties'' (the GEAR Network). We also thank the referee for a very careful reading that caught an omission in the original draft.

%---------------------------------------------------------------------------------
%%% Construction of arithmetic lattices
%---------------------------------------------------------------------------------
\section{Constructing arithmetic lattices}\label{section:construction}

We begin by defining some notation. Throughout this paper $K$ and $K_0$ will denote number fields with rings of integers $\mathcal O_K$ and $\mathcal O_{K_0}$. Given a quaternion algebra $B$ over $K$, $\Ram(B)$ denotes the set of primes of $K$, both finite and infinite, that ramify in $B$. Then $\Ram_\infty(B)$ (resp.\ $\Ram_f(B)$) denotes the set of archimedean places of $K$ (resp.\ finite primes of $K$) that ramify in $B$. Given a prime $\frakp\in\Ram_f(B)$, we denote by $\inv_\frakp(B)$ the local Hasse invariant of $B$ at the prime $\frakp$ (i.e., the Hasse invariant of the $K_\frakp$-algebra $B\otimes_K K_\frakp$).

We now recall the construction of arithmetic lattices acting on products of hyperbolic planes and $3$-dimensional hyperbolic spaces. For a more detailed treatment see \cite[\S 3]{borel}. 

Let $K$ be a number field with $r_1$ real places and $r_2$ complex places, so $[K:\bfQ]=r_1+2r_2$, and $B$ be a quaternion algebra over $K$ that is not totally definite. The isomorphism
\[
B\otimes_{\bfQ} \bfR \cong \mathrm{M}_2(\bfR)^s \times \mathbb H^r \times \mathrm{M}_2(\bfC)^{r_2},
\]
where $s+r=r_1$, determines an embedding
\[
\pi: B^\times \hookrightarrow \prod_{v_i\not\in\Ram_\infty(B)} (B\otimes_K K_{v_i})^\times \longrightarrow \GL_2(\bfR)^s \times \GL_2(\bfC)^{r_2}.
\]
Restricting to the elements $B^1$ of $B^\times$ with reduced norm $1$ induces an embedding
\[
\pi: B^1\hookrightarrow \SL_2(\bfR)^s \times \SL_2(\bfC)^{r_2}.
\]

Let $\mathcal O$ be a maximal order of $B$ and $\mathcal O^1$ be the multiplicative subgroup of $\mathcal O^\times$ consisting of those elements with reduced norm $1$. The image $\pi(\mathcal O^1)\subset \SL_2(\bfR)^s \times \SL_2(\bfC)^{r_2}$ is a lattice by the work of Borel and Harish-Chandra \cite{BHC}. An irreducible lattice $\Gamma \subset \SL_2(\bfR)^a \times \SL_2(\bfC)^{b}$ is arithmetic if and only if it is commensurable with a lattice of the form $\pi(\mathcal O^1)$. When $a + b > 1$, all irreducible lattices are in fact arithmetic \cite[p.\ 2]{margulis}. We note that two arithmetic lattices are commensurable if and only if they are associated with the same number field $K$ and quaternion algebras over $K$ that are $\Aut(K/\bfQ)$-isomorphic \cite[Theorem 8.4.7]{MR}. Lastly, an arithmetic lattice is cocompact if and only if its associated quaternion algebra is not isomorphic to the matrix algebra $\mathrm{M}_2(K)$.

%---------------------------------------------------------------------------------
%%% Main Theorem
%---------------------------------------------------------------------------------
\section{Lattices acting on products of hyperbolic spaces}

Before stating this section's main theorem we introduce convenient notation for dealing with extensions of archimedean places of number fields. Let $K_0$ be a number field and $v$ an archimedean place of $K_0$. Given a finite degree extension $K$ of $K_0$, we define $r_K(v)$ to be the number of extensions of $v$ to $K$ that are real and $c_K(v)$ to be the number of extensions of $v$ to $K$ that are complex. Thus $[K:K_0]=r_K(v)+2c_K(v)$. As an example, if $K$ is a cubic extension of $\bfQ$ whose minimal polynomial has exactly one real root and $v_\infty$ denotes the real place of $\bfQ$, then $r_K(v_\infty)=1$ and $c_K(v_\infty)=1$.

%---------------------------------------------------------------------------------
\begin{thm}\label{theorem:sublattices}
Let $\Gamma\subset \SL_2(\bfR)^a\times \SL_2(\bfC)^b$ be an irreducible arithmetic lattice arising from a quaternion algebra $A$ over a number field $K$. There is a one-to-one correspondence between
\begin{enumerate}

\item commensurability classes of irreducible arithmetic lattices of $\SL_2(\bfR)^c\times \SL_2(\bfC)^d$ that are contained in $\Gamma$, and

\item subfields $K_0$ of $K$ having precisely $d$ complex places and $\Aut(K_0/\bfQ)$-isomorphism classes of $K_0$-subalgebras $B\subset A$ such that $A\cong B\otimes_{K_0}K$, where $B$ is split at precisely $c$ real places of $K_0$.

\end{enumerate}
Furthermore, under this correspondence it is necessarily the case that
\[
[K:K_0]=\frac{2b+a+\sum_{v_j\in\Ram_\infty(B)} r_K(v_j)}{2d+c+\#\Ram_\infty(B)}.
\]
\end{thm}
%---------------------------------------------------------------------------------

%---------------------------------------------------------------------------------
\begin{rmk}
When $a+b>1$ it follows from the Margulis arithmeticity theorem \cite[p.\ 2]{margulis} that all irreducible lattices in $\SL_2(\bfR)^a\times \SL_2(\bfC)^b$ are arithmetic, hence the word ``arithmetic'' may be removed from the first statement in Theorem \ref{theorem:sublattices}.
\end{rmk}
%---------------------------------------------------------------------------------
	
Before proving Theorem \ref{theorem:sublattices} we consider a few examples in which Theorem \ref{theorem:sublattices} simplifies considerably.

%---------------------------------------------------------------------------------
\begin{example}
Consider the case in which $(a,b)=(0,1)$ and $(c,d)=(1,0)$. In this case Theorem \ref{theorem:sublattices} provides a criterion for the existence of arithmetic Fuchsian subgroups of a fixed arithmetic Kleinian group. Consider a real place $v$ of $K_0$ that ramifies in $B$. If $v$ extends to a complex place of $K$ then the fact that $K$ must have a unique complex place implies that any real place of $K_0$ splitting in $B$ extends to at least two real places of $K$ which necessarily split in $A$. This contradicts that fact that $A$ is ramified at all real places of $K$ (implied by our hypothesis that $a=0$). We therefore deduce that $r_K(v_j)=[K:K_0]$ for all $v_j\in\Ram_\infty(B)$, hence the above formula for $[K:K_0]$ shows that $[K:K_0]=2$.  Theorem \ref{theorem:sublattices} therefore shows that commensurability classes of arithmetic Fuchsian subgroups of $\Gamma$ correspond to degree $2$ totally real subfields $K_0$ of $K$ and isomorphism classes of quaternion algebras $B$ over $K_0$, split at a unique real place, satisfying $A\cong B\otimes_{K_0}K$. This is Theorem 9.5.4 of \cite{MR}.
\end{example}
%---------------------------------------------------------------------------------

%---------------------------------------------------------------------------------
\begin{example}\label{example:totallyreal}
Now consider the case in which $(a,b)=(n,0)$ and $(c,d)=(m,0)$, where $n>m\geq 1$. In this situation both $K$ and $K_0$ are totally real fields. It follows that every real place of $K_0$ extends to $[K:K_0]$ real places of $K$ so that 
$r_K(v_j)=[K:K_0]$ for all $v_j\in\Ram_\infty(B)$. Thus $[K:K_0]=\frac{n}{m}$. In particular this implies that if an irreducible lattice in $\SL_2(\bfR)^n$ contains an irreducible arithmetic lattice in $\SL_2(\bfR)^m$ then $m\mid n$ and $n\geq 2m$. When $n=2$ and $m=1$ this reproves Theorem 3.1 of \cite{CS}.
\end{example}
%---------------------------------------------------------------------------------

We now prove Theorem \ref{theorem:sublattices}.

%---------------------------------------------------------------------------------
\begin{proof}[Proof of Theorem \ref{theorem:sublattices}]
That a subfield $K_0$ of $K$ and quaternion algebra $B$ as in (ii) will produce a commensurability class of arithmetic subgroups of $\SL_2(\bfR)^c\times \SL_2(\bfC)^d$ contained in $\Gamma$ is clear.

Conversely, suppose that $\Sigma$ is an irreducible arithmetic lattice of $\SL_2(\bfR)^c\times \SL_2(\bfC)^d$ contained in $\Gamma$. Denote by $B$ the quaternion algebra associated to $\Sigma$ and by $K_0$ the center of $B$. That $K_0$ has exactly $d$ complex places is clear. The inclusion $\Sigma\subset \Gamma$ induces an embedding of algebras $B\hookrightarrow A$, and it follows immediately that $A\cong B\otimes_{K_0}K$ and that $B$ splits at precisely $c$ real places of $K_0$.

All that remains is to prove that
\begin{equation}\label{formula}
[K:K_0]=\frac{2b+a+\sum_{v_j\in\Ram_\infty(B)} r_K(v_j)}{2d+c+\#\Ram_\infty(B)}.
\end{equation}
Neither a complex place of $K_0$ nor a real place of $K_0$ ramifying in $B$ can extend to a real place of $K$ that is split in $A$. Indeed, in the latter case we would obtain an induced embedding of Hamilton's quaternion algebra into $\mathrm{M}_2(\bfR)$, which is a contradiction. We conclude that only real places of $K_0$ that split in $B$ can extend to real places of $K$ splitting in $A$. Therefore
\begin{align}
\label{realplaces}a&=\sum_{v_i\not\in\Ram_\infty(B)} r_K(v_i) \\
\label{complexplaces}b&=d[K:K_0]+\sum_{v_i} c_K(v_i),
\end{align}
where the latter sum is taken over the real places of $K_0$. Formula \eqref{formula} now follows from combining the equality $[K:K_0]=r_K(v_i)+2c_K(v_i)$ with equation \eqref{complexplaces} and substituting \eqref{realplaces} into the resulting expression.
\end{proof}	
%---------------------------------------------------------------------------------

In concrete applications, one needs a criterion for a $K_0$-quaternion algebra $B$ to embed in a $K$-algebra $A$. The following lemma gives necessary and sufficient criteria for $B$ to embed in $A$. This generalizes Theorem 9.5.5 of \cite{MR}, which considers the case in which $[K:K_0]=2$.
%---------------------------------------------------------------------------------
\begin{lem}\label{lem:WhenEmbed?}
Let $A$ be a $K$-quaternion algebra and $K_0 \subset K$ a proper subfield. Then there exists a $K_0$-quaternion algebra $B$ such that $A \cong B \otimes_{K_0} K$ if and only if for every $w \in \Ram(A)$ and $v$ the unique place of $K_0$ divisible by $w$, the following conditions hold.
\begin{enumerate}
\item The local degree $[K_w : (K_0)_v]$ is odd.
\item If $w'\neq w$ is place of $K$ dividing $v$ and the local degree $[K_{w'} : (K_0)_v]$ is odd then $w'\in\Ram(A)$.
\item If every place $v'$ of $K_0$ is divisible by a place $w'$ of $K$ for which $[K_{w'} : (K_0)_{v'}]$ is odd, then \[\{v : v \text{ is a place of $K_0$ divisible by some } w\in\Ram(A)\}\] has even cardinality.
\end{enumerate}
\end{lem}

\begin{proof}
Suppose first that $B$ is a $K_0$-quaternion algebra such that $A \cong B \otimes_{K_0} K$. Let $w \in \Ram(A)$ and $v$ the unique place of $K_0$ divisible by $w$. By \cite[Thm.\ 31.9]{reiner} we see that
\[
\frac{1}{2}=\inv_w(A)=\inv_w(B \otimes_{K_0} K)=[K_w : (K_0)_v]\cdot\inv_v(B)\in\bfQ/\bfZ,
\]
hence we conclude that the local degree $[K_w : (K_0)_v]$ is odd and that $B$ is ramified at $v$. Now let $w'$ be another place of $K$ dividing $v$ and assume that $[K_{w'} : (K_0)_v]$ is odd. Then
\[
\inv_{w'}(A)=\inv_{w'}(B \otimes_{K_0} K)=[K_{w'} : (K_0)_v]\cdot\inv_v(B)=\inv_v(B)=\frac{1}{2}\in\bfQ/\bfZ,
\]
since $v$ ramifies in $B$. This shows that $w'\in\Ram(A)$.

We now show that (iii) holds. To that end, suppose that every place $v'$ of $K_0$ is divisible by some place $w'$ of $K$ with $[K_{w'} : (K_0)_{v'}]$ odd. The arguments in the previous paragraph show that
\[
\Ram(B)\supseteq \{v\ :\ v \text{ is divisible by some } w\in\Ram(A)\}.
\]
We will show that this is in fact an equality. To see this, let $v'$ be a place of $K_0$ that is in $\Ram(B)$ but is not divisible by a place of $K$ ramifying in $A$. By assumption there is a place $w'$ of $K$ which divides $v'$ and for which $[K_{w'} : (K_0)_{v'}]$ is odd. Observe that $\inv_{w'}(B\otimes_{K_0} K)=[K_{w'} : (K_0)_{v'}]\cdot\inv_{v'}(B)$. Since $w'\not\in \Ram(A)$ we necessarily have that $[K_{w'} : (K_0)_{v'}]$ is even, which is a contradiction.

Conversely, suppose that $w,w'$ and $v$ are as in the statement of the lemma and that conditions (i), (ii) and (iii) are satisfied. Let $B$ be the quaternion algebra over $K_0$ for which
\[
\Ram(B)=\{v\ :\ v \text{ is divisible by some } w\in\Ram(A)\}\cup S,
\]
where the set $S$ is empty if $\{v : v \text{ is divisible by some } w\in\Ram(A)\}$ has even cardinality and consists of a single place $v'$ of $K_0$ for which every place $w'$ of $K$ dividing $v'$ satisfies $[K_{w'} : (K_0)_{v'}]$ even otherwise. We note that such an algebra $B$ exists because condition (iii) implies that such a $v'$ exists whenever $S$ is required to be non-empty in order to have a set
\[
\{v : v \text{ is divisible by some } w\in\Ram(A)\}\cup S
\]
of even cardinality.

To show that $A \cong B \otimes_{K_0} K$ we must show that $\Ram(A)=\Ram(B \otimes_{K_0} K)$ (see e.g., \cite[\S 31]{reiner} and \cite[p.\ 277]{reiner}). Suppose first that $w\in\Ram(A)$ and let $v$ be the place of $K_0$ divisible by $w$. Then
\[
\inv_w(B \otimes_{K_0} K)=[K_w : (K_0)_v]\cdot \inv_v(B)=\frac{1}{2}\in\bfQ/\bfZ
\]
by condition (i) and the fact that $v\in\Ram(B)$. Thus $w\in\Ram(B \otimes_{K_0} K)$.

Now suppose that $w'\in\Ram(B \otimes_{K_0} K)$ and that $v$ is the place of $K_0$ divisible by $w'$. Then $\inv_{w'}(B \otimes_{K_0} K)=[K_{w'} : (K_0)_v]\inv_v(B)$, hence $B$ is ramified at $v$ and $[K_{w'} : (K_0)_v]$ is odd. By definition of $\Ram(B)$, $v$ is divisible by some place $w$ of $K$ lying in $\Ram(A)$. We now conclude that $w'\in\Ram(A)$ from condition (ii). This shows that $\Ram(A)=\Ram(B \otimes_{K_0} K)$ and finishes the proof.
\end{proof}
%---------------------------------------------------------------------------------

%---------------------------------------------------------------------------------
%%% classes of sublattices
%---------------------------------------------------------------------------------
\section{Commensurability classes of arithmetic sublattices}

It is well-known that if an arithmetic hyperbolic $3$-manifold contains one totally geodesic surface, then it contains infinitely many distinct commensurability classes of totally geodesic surfaces. In \cite[Theorem 1.1]{CS} it was shown that if a Shimura surface contains one geodesic curve, then it contains infinitely many that are pairwise incommensurable. The following result gives the appropriate extension to higher rank arithmetic lattices, showing that this phenomenon is closely related to the fact that, in both cases, submanifolds are associated with Galois extensions of even degree.

%---------------------------------------------------------------------------------
\begin{thm}\label{theorem:infinitelymanysublattices}
Let $K$ be a number field and $K_0\subset K$ a subfield for which $K/K_0$ is Galois of even degree. If $\Gamma\subset \SL_2(\bfR)^a\times\SL_2(\bfC)^b$ is an irreducible arithmetic lattice arising from a quaternion algebra $A$ over $K$ that contains an irreducible arithmetic sublattice $\Sigma\subset \SL_2(\bfR)^c\times \SL_2(\bfC)^d$ arising from a quaternion algebra $B$ over $K_0$, then $\Gamma$ contains infinitely many pairwise incommensurable arithmetic lattices in $\SL_2(\bfR)^c\times \SL_2(\bfC)^d$.
\end{thm}
%---------------------------------------------------------------------------------

Before proving this, we need the following lemma.

%---------------------------------------------------------------------------------
\begin{lem}\label{lem:EvenLocal}
Let $K / K_0$ be a finite Galois extension of number fields with even degree. There exist infinitely many places $\mathfrak{p}$ of $K_0$ that are unramified in $K$ such that the local extension $K_{\mathfrak{q}} / (K_0)_{\mathfrak{p}}$ has even degree for every prime $\mathfrak{q}$ of $K$ over $\mathfrak{p}$.
\end{lem}
%---------------------------------------------------------------------------------

%---------------------------------------------------------------------------------
\begin{proof}
Let $G = \Gal(K / K_0)$. Since $G$ has even order, it contains an element $\sigma$ of order two. It follows from the Cebotarev density theorem that there are infinitely many primes $\mathfrak{p}$ of $K_0$ such that every prime $\mathfrak{q}$ of $K$ above $\mathfrak{p}$ has Frobenius element a conjugate of $\sigma$. Since only finitely many primes of $K_0$ ramify in $K$, we can always assume that $\mathfrak{p}$ is unramified. With these assumptions we have $[K_{\mathfrak{q}} : (K_0)_{\mathfrak{p}}] = 2$ for every $\mathfrak{q}$ over $\mathfrak{p}$. The lemma follows.
\end{proof}
%---------------------------------------------------------------------------------

We now prove Theorem \ref{theorem:infinitelymanysublattices}.

%---------------------------------------------------------------------------------
\begin{proof}[Proof of Theorem \ref{theorem:infinitelymanysublattices}]
Let $\Sigma, \Gamma$ be as in the statement of the theorem. It follows from the theory developed in \cite{MR} (see also the discussion in the proof of \cite[Thm.\ 3.1]{CS}) that the inclusion $\Sigma \subset \Gamma$ induces an embedding of quaternion algebras $B\hookrightarrow A$. It follows that $A\cong B\otimes_{K_0}K$. By Theorem \ref{theorem:sublattices}, it suffices to show that there exist infinitely many quaternion algebras $B'$ over $K_0$, that are pairwise $\Aut(K_0/\bfQ)$-non-conjugate, have the same archimedean ramification as $B$, and all satisfy $A\cong B'\otimes_{K_0}K$.

By Lemma \ref{lem:EvenLocal}, there are infinitely many primes $\frakp$ of $K_0$ for which $[K_{\mathfrak{q}} : (K_0)_{\mathfrak{p}}]$ is even degree for every prime ideal $\mathfrak{q}$ of $K$ dividing $\mathfrak{p}$ and $\mathfrak{p}$ is totally unramified in $K$. Pick two such primes $\mathfrak{p}_1$ and $\mathfrak{p}_2$ of $K_0$.

Let $B'$ be the quaternion algebra over $K_0$ for which $\Ram(B')=\Ram(B)\cup \{\frakp_1,\frakp_2\}$. Note that by Reiner \cite[Thm.\ 31.9]{reiner}, the Hasse invariant of $B'\otimes_{K_0} K$ at a prime $\frakq_i$ of $K$ lying above $\frakp_i$ satisfies
\[
\inv_{\frakq_{i}}(B'\otimes_{K_0} K) = [K_{\frakq_i}:(K_0)_{\frakp_i}] \inv_{\frakp_i}(B') = 1 \in  \bfQ / \bfZ,
\]
since $[K_{\frakq_i}:(K_0)_{\frakp_i}]$ is even. This shows that $B'\otimes_{K_0} K$ and $B\otimes_{K_0}K$ are ramified at precisely the same places of $K$, and hence are both isomorphic to $A$. There are infinitely many possibilities for $\frakp_1,\frakp_2$ and hence for $B'$, which completes the proof of the theorem.
\end{proof}
%---------------------------------------------------------------------------------

The following is an easy consequence of Theorems \ref{theorem:sublattices} and \ref{theorem:infinitelymanysublattices}.

%---------------------------------------------------------------------------------
\begin{cor}\label{cor:2^n}
%Let $K/\bfQ$ be a totally real cyclic extension of even degree $2 n$. 
If $K$ is an abelian extension of $\bfQ$ and $\Gamma\subset \SL_2(\bfR)^{2 n}$ is an irreducible lattice arising from a quaternion algebra over $K$ and that contains an arithmetic sublattice $\Sigma\subset \SL_2(\bfR)^m$ for some $m\vert n$, then $\Gamma$ contains infinitely many pairwise incommensurable arithmetic sublattices arising from some embedding of $\SL_2(\bfR)^m$ into $\SL_2(\bfR)^{2 n}$.
\end{cor}
%---------------------------------------------------------------------------------

We conclude this section by showing that the phenomenon described above, in which an arithmetic lattice containing a single commensurability class of arithmetic sublattices with positive codimension in fact contains infinitely many, is not true in general. We first study the case where $K / K_0$ is even degree, but not Galois.

%---------------------------------------------------------------------------------
\begin{thm}\label{thm:A5}
There exist infinitely many commensurability classes of irreducible arithmetic lattices acting on $(\mathbf{H}^2)^6$ containing exactly one commensurability class of arithmetic Fuchsian subgroups and no other arithmetic sublattices.
\end{thm}
%---------------------------------------------------------------------------------

%---------------------------------------------------------------------------------
\begin{proof}
Let $K$ be the field with minimal polynomial
\[
f(t) = t^6 - 10 t^4 + 7 t^3 + 15 t^2 - 14 t + 3.
\] This is a totally real sextic extension of $\bfQ$ with Galois group $A_5$. If $L / \bfQ$ is the Galois closure, then
\begin{align*}
\Gal(L / \bfQ) &\cong A_5 \\
\Gal(L / K) &\cong D_5.
\end{align*}

Let $p$ be a rational prime that is unramified in $K$, and $\mathfrak{p}_1, \dots, \mathfrak{p}_r$ be the primes of $K$ above $p$. Considering the conjugacy classes of elements in $A_5$ and the intersection of each conjugacy class with $D_5$, we only have the following possibilities for the local completions:
\begin{enumerate}

\item $r = 6$ and
\begin{align*}
[K_{\mathfrak{p}_i} : \bfQ_p] = 1& &i = 1,\dots,6
\end{align*}

\item $r = 4$ and
\begin{align*}
[K_{\mathfrak{p}_i} : \bfQ_p] = 1& &i = 1, 2 \\
[K_{\mathfrak{p}_i} : \bfQ_p] = 2& &i = 3, 4
\end{align*}

\item $r = 2$ and
\begin{align*}
[K_{\mathfrak{p}_1} : \bfQ_p] = 1 \\
[K_{\mathfrak{p}_2} : \bfQ_p] = 5
\end{align*}

\item $r = 2$ and
\[
[K_{\mathfrak{p}_i} : \bfQ_p] = 3 \quad i = 1, 2
\]

\end{enumerate}

In particular, we see that for every rational prime $p$ that is unramified in $K$, there exists a prime $\mathfrak{p}$ of $K$ above $p$ with $[K_{\mathfrak{p}} : \bfQ_p]$ of odd degree. Recall from the proof of Theorem \ref{theorem:infinitelymanysublattices} that if $B$ is a quaternion algebra over $\bfQ$ and $A = B \otimes_\bfQ K$, then
\[
\inv_{\frakp}(A) = [K_{\frakp}:\bfQ_p] \inv_p(B) \in \frac{1}{2} \bfZ / \bfZ.
\]
Consequently, if $p \in \mathrm{Ram}(B)$ is a rational prime that is unramified in $K$, then there exists a prime $\mathfrak{p}$ of $K$ above $p$ such that $A$ ramifies at $\mathfrak{p}$.

Given a finite set $\mathcal{S}$ of rational primes with even cardinality, containing no primes that ramify in $K$, let $B_{\mathcal{S}}$ denote the unique $\bfQ$-quaternion algebra that ramifies at precisely the places in $\mathcal{S}$ and
\[
A_{\mathcal{S}} = B_{\mathcal{S}} \otimes_\bfQ K.
\]
Notice that the pair $(\mathbb Q, B_{\mathcal{S}})$ determines a commensurability class of arithmetic Fuchsian groups.

Suppose that $B'$ is another $\bfQ$-quaternion algebra such that $A_{\mathcal{S}}  \cong B' \otimes_\bfQ K$. Then it is not hard to show that $\mathcal{S} \subseteq \mathrm{Ram}(B')$. Indeed, every prime of $K$ at which $A_{\mathcal{S}}$ ramifies lies above a prime in $\mathcal{S}$. It is also not hard to see that the infinite place $v_\infty$ of $\bfQ$ cannot be in $\mathrm{Ram}(B')$, since $K$ is totally real and $A_{\mathcal{S}}$ is unramified at all the real places of $K$. In particular, we see that if $p \in \mathrm{Ram}(B') \smallsetminus \mathcal{S}$, then $p$ ramifies in $K$.

With the choice of $f(t)$ as above, the rational primes that ramify in $K$ are $19$ and $293$. Since $\mathrm{Ram}(B')$ has even cardinality, we see that either $\mathrm{Ram}(B') = \mathcal{S}$, hence $B \cong B'$, or
\[
\mathrm{Ram}(B') = \mathcal{S} \cup \{19, 293\}.
\]
However, there is a prime $\mathfrak{p}$ of $K$ above $293$ of degree $1$, i.e., such that $K_{\mathfrak{p}} \cong \bfQ_{293}$, and it follows that $B' \otimes_{\bfQ} K$ also ramifies at $\mathfrak{p}$, which is a contradiction.

To recap, let $K$ be the totally real sextic extension of $\bfQ$ with minimal polynomial $f(t)$ as above and $\mathcal{S}$ any finite set of rational primes of even cardinality not containing $19$ or $293$. Let $B_{\mathcal{S}}$ be the unique $\bfQ$-quaternion algebra ramified exactly at the primes in $\mathcal{S}$, and $A_{\mathcal{S}}$ be the base change of $B_{\mathcal{S}}$ from $\bfQ$ to $K$. We have shown
\[
A_{\mathcal{S}} \cong B' \otimes_{\bfQ} K
\]
for some $\bfQ$-quaternion algebra $B'$ if and only if $B' \cong B_{\mathcal{S}}$.

Then $A_{\mathcal{S}}$ determines a commensurability class of arithmetic lattices in $\SL_2(\bfR)^6$, and the above proves that it contains a unique commensurability class of arithmetic Fuchsian subgroups. To see that there are no other arithmetic sublattices, it suffices to notice that $K$ contains no proper subfield other than $\bfQ$. This completes the proof.
\end{proof}
%---------------------------------------------------------------------------------

We now study the odd-degree cyclic Galois case.

%---------------------------------------------------------------------------------
\begin{thm}\label{theorem:finitelymany}
Let $K/\bfQ$ be a totally real cyclic Galois extension of odd degree $n>1$. There exist infinitely many pairwise incommensurable irreducible lattices $\Gamma\subset \SL_2(\bfR)^n$ with invariant trace field $K$ whose arithmetic sublattices lie in precisely $\tau(n)$ commensurability classes, where $\tau$ denotes the divisor function.
\end{thm}
%---------------------------------------------------------------------------------

%---------------------------------------------------------------------------------
\begin{proof}
Let $p,q$ be distinct primes that split completely in $K/\bfQ$, let $\frakp_1,\dots,\frakp_n$ be the prime divisors in $K$ of $p\mathcal O_K$, and let $\frakq_1,\dots,\frakq_n$ be the prime divisors in $K$ of $q\mathcal O_K$. Consider the quaternion algebra $A$ over $K$ that ramifies precisely at the primes $\frakp_1,\dots,\frakp_n,\frakq_1,\dots,\frakq_n$. If $\mathcal{O}$ is a maximal order of $A$, then $\Gamma = \pi(\mathcal O^1)$ is an irreducible lattice in $\SL_2(\bfR)^{n}$. We will show that the arithmetic sublattices in $\Gamma$ with positive codimension belong to precisely $\tau(n)-1$ commensurability classes. Since infinitely many primes split completely in $K/\bfQ$, our assertion about the existence of infinitely many incommensurable arithmetic lattices $\Gamma$ with the requisite properties will follow.

Example \ref{example:totallyreal} shows that any positive codimension arithmetic sublattice of $\Gamma$ lies in $\SL_2(\bfR)^d$ for some divisor $d$ of $n$. Fix one such divisor $d$ and observe that there is a unique subfield $K_0$ of $K$ for which $[K:K_0]=\frac{n}{d}$. By Theorem \ref{theorem:sublattices} it suffices to show that there is a unique quaternion algebra $B$ defined over $K_0$ such that $A\cong B\otimes_{K_0}K$.

To begin, let $B$ be the quaternion algebra over $K_0$ that is unramified at all archimedean places of $K_0$ and in which all prime divisors of $p\mathcal O_{K_0}$ and $q\mathcal O_{K_0}$ ramify. To show that $B\otimes_{K_0}K$ is isomorphic to $A$ we must show that the only nontrivial Hasse invariants of $B\otimes_{K_0}K$ are associated with the prime divisors of $p\mathcal O_{K}$ and $q\mathcal O_{K}$. Let $\frakp$ be a prime divisor of $p\mathcal O_{K_0}$ or $q\mathcal O_{K_0}$ and $\frakP$ be a prime of $K$ lying above $\frakp$. Since $p$ and $q$ both split completely in $K/\bfQ$, it follows that $\frakp$ splits completely in $K/K_0$, hence $[K_\frakP:(K_0)_\frakp]=1$. We now conclude that
\[
\inv_\frakP(B\otimes_{K_0}K)=[K_\frakP:(K_0)_\frakp]\inv_{\frakp}(B)=1\cdot \frac{1}{2}=\frac{1}{2} \in \frac{1}{2} \bfZ / \bfZ.
\]
Conversely, let $\mathfrak{L}$ be a prime of $K$ not lying above $p$ or $q$ and let $\mathfrak{l}=\mathfrak{L}\cap K_0$. Then
\[
(B\otimes_{K_0}K)\otimes_K K_\mathfrak{L}\cong(B\otimes_{K_0}K_\mathfrak{l})\otimes_{K_\mathfrak{l}} K_\mathfrak{L}\cong \mathrm{M}_2(K_\mathfrak{l})\otimes_{K_\mathfrak{l}} K_\mathfrak{L}\cong \mathrm{M}_2(K_\mathfrak{L}),
\]
showing that $B\otimes_{K_0}K$ is unramified at $\mathfrak{L}$. This shows that $A\cong B\otimes_{K_0}K$.

Now, suppose that $B'$ is a quaternion algebra over $K_0$ such that $A\cong B'\otimes_{K_0}K$. No real place of $K_0$ that ramifies in $B'$ can extend to a real place of $K$ splitting in $B'\otimes_{K_0}K$, as this would induce an embedding of Hamilton's quaternions into $\mathrm{M}_2(\bfR)$. Since $K$ (and hence $K_0$) is totally real and $A$ is split at all archimedean places, it follows that $B'$ is split at all archimedean places of $K_0$. If $\frakp$ is a prime of $K$ lying above $p$ or $q$ then the above argument shows that $B'$ must be ramified at $\frakp$.

We now show that if $\mathfrak{l}$ is a prime of $K_0$ not lying above $p$ or $q$ then $\mathfrak{l}$ is unramified in $B'$. Suppose to the contrary that such a prime $\mathfrak{l}$ ramifies in $B'$ and let $\mathfrak L$ be a prime of $K$ lying above $\mathfrak{l}$. Note that $\inv_\mathfrak{L}(A)=1$. Then
\[
1=\inv_\mathfrak{L}(A)=\inv_\mathfrak{L}(B'\otimes_{K_0}K)=[K_{\mathfrak{L}}:(K_0)_{\mathfrak{l}}]\inv_{\mathfrak l}(B')=[K_{\mathfrak{L}}:(K_0)_{\mathfrak{l}}]\cdot \frac{1}{2},
\]
forcing us to deduce that
\[
[K_{\mathfrak{L}}:(K_0)_{\mathfrak{l}}]=e(\mathfrak{L}/\mathfrak{l})f(\mathfrak{L}/\mathfrak{l})
\]
is even. However, this is impossible as $e(\mathfrak{L}/\mathfrak{l})f(\mathfrak{L}/\mathfrak{l})$ is a divisor of $\frac{n}{d}$ (since $K/K_0$ is Galois), which by hypothesis must be odd. We conclude that $B$ and $B'$ are ramified at precisely the same primes of $K_0$, hence they are isomorphic. It follows that $B$ is the unique quaternion algebra over $K_0$ for which $A\cong B\otimes_{K_0}K$, finishing the proof.
\end{proof}	
%---------------------------------------------------------------------------------

The following is an immediate consequence of Theorem \ref{theorem:finitelymany} and the fact that for every $m\geq 1$ there is an integer $n\geq 1$ such that $\tau(n)=m$.

%---------------------------------------------------------------------------------
\begin{cor}
For every positive integer $m\geq 1$ there exist infinitely many pairwise incommensurable irreducible lattices containing precisely $m$ commensurability classes of arithmetic sublattices.
\end{cor}
%---------------------------------------------------------------------------------

We close with the following, which is well-known to experts.

%---------------------------------------------------------------------------------
\begin{lem}\label{lem:LotsOfDistinct}
Let $G$ be a semisimple Lie group of noncompact type with associated symmetric space $X$ and $Y \subset X$ be an embedded proper totally geodesic subspace associated with the subgroup
\[
\mathrm{Stab}_G(Y)= H \subset G.
\]
Suppose that $\Gamma$ is an arithmetic lattice in $G$ such that $\Lambda = \Gamma \cap H$ is a lattice in $H$. Then $X / \Gamma$ contains infinitely many distinct immersed totally geodesic submanifolds of the form $Y / \Lambda^\prime$ with $\Lambda^\prime$ a lattice in $H$ commensurable with $\Lambda$.
\end{lem}
%---------------------------------------------------------------------------------

%---------------------------------------------------------------------------------
\begin{proof}
Suppose that $M = X / \Gamma$ and that $N = Y / \Lambda$ admits an immersion as a proper totally geodesic submanifold of $M$, where $\Gamma$ (resp.\ $\Lambda$) is a lattice in $G$ (resp.\ $H$) as in the statement of the lemma. Then $N$ is nowhere dense in $M$, and the lift of $N$ to the universal cover $\widetilde{M} = X$ is the $\Gamma$-orbit of $Y$, which represents a chosen lift. In particular, notice that $Y \cdot \Gamma$ is also nowhere dense in $X$.

Without loss of generality, we can replace $G$ with its adjoint group. Recall that, since $\Gamma$ is arithmetic, the Margulis dichotomy \cite{margulis} then implies that the \emph{commensurator}
\[
\mathcal{C}(\Gamma) = \{g \in G~:~\Gamma\ \textrm{is commensurable with}\ g \Gamma g^{-1} \}
\]
is analytically dense in $G$. Given $c \in \mathcal{C}(\Gamma)$, define
\[
\Lambda_c^0 = c \Lambda c^{-1} \cap \Gamma,
\]
which is a lattice in $c H c^{-1} \subset G$. Then $\Lambda_c = \Gamma \cap c H c^{-1}$ is a sublattice of $\Gamma$ containing $\Lambda_c^0$, and hence determines an immersion of $N_c = c(Y) / \Lambda_c$ as a totally geodesic subspace of $M$, where $c(Y) = Y \cdot c^{-1}$. By construction, $\Lambda_c$ is commensurable with $\Lambda$. From the density of $\mathcal{C}(\Gamma)$ in $G$ and nowhere density of $Y \cdot \Gamma$ in $X$, we can choose $c$ such that $c(Y)$ is not contained in $\Gamma \cdot Y$. It follows immediately that $N_c$ determines a geodesic submanifold of $M$ distinct from $N$. Continuing inductively produces arbitrarily many distinct geodesic subspaces of $M$, which proves the lemma.
\end{proof}
%---------------------------------------------------------------------------------

%---------------------------------------------------------------------------------
\begin{rmk}
The spaces $N_i$ constructed in the proof of Lemma \ref{lem:LotsOfDistinct} are almost certainly not homeomorphic, and have arbitrarily large volume as $i \to \infty$ (and hence are likely to be quite far from embedded). For some results along these lines, see \cite{MRBianchi, Toledo, MK}.
\end{rmk}
%---------------------------------------------------------------------------------

%---------------------------------------------------------------------------------

%---------------------------------------------------------------------------------
%---------------------------------------------------------------------------------

\end{document}